\documentclass[a4,12pt]{amsart}
\usepackage{amssymb,amstext,amsmath,amscd,amsthm,
amsfonts,enumerate,graphicx,latexsym}
\usepackage[usenames]{color}
\usepackage[all]{xy}

\usepackage{url}
\newtheorem{theorem}{Theorem}[section]
\newtheorem{corollary}[theorem]{Corollary}
\newtheorem{lemma}[theorem]{Lemma}
\newtheorem{proposition}[theorem]{Proposition}

\theoremstyle{definition}

\newtheorem{remark}[theorem]{Remark}

\newtheorem*{question*}{Question}

\newtheorem*{conjecture*}{Conjecture}
\newtheorem{example}[theorem]{Example}

\newtheorem*{notation*}{Notation}

\newtheorem*{claim*}{Claim}
\newtheorem{definitiontheorem}[theorem]{Definition-Theorem}
\numberwithin{equation}{theorem}

\def\Hom{\operatorname{Hom}}

\def\End{\operatorname{End}}

\def\op{{\rm op}}

\def\Kb{\mathsf{K}^{\rm b}}
\def\thick{\operatorname{\mathsf{thick}}}
\def\add{\operatorname{\mathsf{add}}}

\def\proj{\operatorname{\mathsf{proj}}}

\def\silt{\operatorname{\mathsf{silt}}}
\def\tilt{\operatorname{\mathsf{tilt}}}
\def\twsilt{\operatorname{\mathsf{2silt}}}

\def\E{\mathcal{E}}

\def\T{\mathcal{T}}

\def\K{\mathcal{K}}

\def\H{\mathcal{H}}
\def\S{\mathbb{S}}

\begin{document}
\setlength{\baselineskip}{15pt}

\title{A symmetry of silting quivers}

\author{Takuma Aihara}
\address{Department of Mathematics, Tokyo Gakugei University, 4-1-1 Nukuikita-machi, Koganei, Tokyo 184-8501, Japan}
\email{aihara@u-gakugei.ac.jp}

\author{Qi Wang}
\address{Yau Mathematical Sciences Center, Tsinghua University, Hai Dian District, Beijing 100084, China}
\email{infinite-wang@outlook.com (cc:infinite-wang@tsinghua.edu.cn)}

\thanks{TA was partly supported by JSPS Grant-in-Aid for Young Scientists 19K14497.
QW was partly supported by JSPS Grant-in-Aid for Young Scientists 20J10492.}

\begin{abstract}
We investigate symmetry of the silting quiver of a given algebra which is induced by an anti-automorphism of the algebra.
In particular, one shows that if there is a primitive idempotent fixed by the anti-automorphism, then the 2-silting quiver ($=$ the support $\tau$-tilting quiver) has a bisection.
Consequently, in that case, we obtain that the cardinality of the 2-silting quiver is an even number (if it is finite).
\end{abstract}
\maketitle
\section*{Introduction}

In this paper, we study symmetry of the silting quiver of a finite dimensional algebra $\Lambda$ over an algebraically closed field; the \emph{silting quiver} is a quiver whose vertices are (basic) silting objects and arrows $T\to U$ are drawn whenever $U$ is an irreducible left mutation of $T$, and it coincides with the Hasse quiver of the poset $\silt\Lambda$ of silting objects \cite{AI}.

The main theorem (Theorem \ref{symmetry}) of this paper shows that an anti-automorphism of $\Lambda$ (i.e., an algebra isomorphism $\Lambda^\op\simeq\Lambda$) induces a symmetry of $\silt\Lambda$.
Here, $\Lambda^\op$ stands for the opposite algebra of $\Lambda$.
Focusing on 2-term silting objects, which bijectively correspond to support $\tau$-tilting modules \cite{AIR}, we obtain a bisection of the poset $\twsilt\Lambda$ of 2-term silting objects if there is a fixed primitive idempotent by the anti-automorphism (Theorem \ref{double}).
Thus, in that case, it turns out that the cardinality of $\twsilt\Lambda$ is even (if it is finite).

When $\Lambda$ is 2-silting finite ($= \tau$-tilting finite); i.e., $|\twsilt\Lambda|<\infty$, counting the number of elements in $\twsilt\Lambda$ is one of the important problems in this area; see \cite{A, AA, AHMW, EJR, Mizuno-preprojective-alg}.
In this context, Theorem \ref{double} gives a very useful method to reduce the whole pattern to half of $\twsilt\Lambda$.
Indeed, this may be applied to such works on Hecke algebras and Schur algebras, see \cite{AS-Hecke}, \cite{W-schur}, etc.

For example, the following admit anti-automorphisms fixing a primitive idempotent:
\begin{itemize}
\item enveloping algebras (Theorem \ref{envelopingalgebra});

\item preprojective algebras of Dynkin type (Theorem \ref{ppas});

\item cellular algebras (Theorem \ref{cellular-alg});

\item symmetric algebras with radical cube zero, which contain multiplicity-free Brauer line/cycle algebras (Theorem \ref{rcz});

\item selfinjective Nakayama algebras, which contain Brauer star algebras with an exceptional vertex in the center (Theorem \ref{Nakayama});

\item group algebras (Theorem \ref{ga});

\item the trivial extensions of algebras with an anti-automorphism fixing a primitive idempotent (Theorem \ref{trivial-extension}).
\end{itemize}
Here is an illustration of the symmetry of $\twsilt\Lambda$ for the preprojective algebra $\Lambda$ of Dynkin type $A_3$,
in which $\xymatrix{*+[F]{-}}$ and $\xymatrix{*+[o][F]{-}}$ correspond:
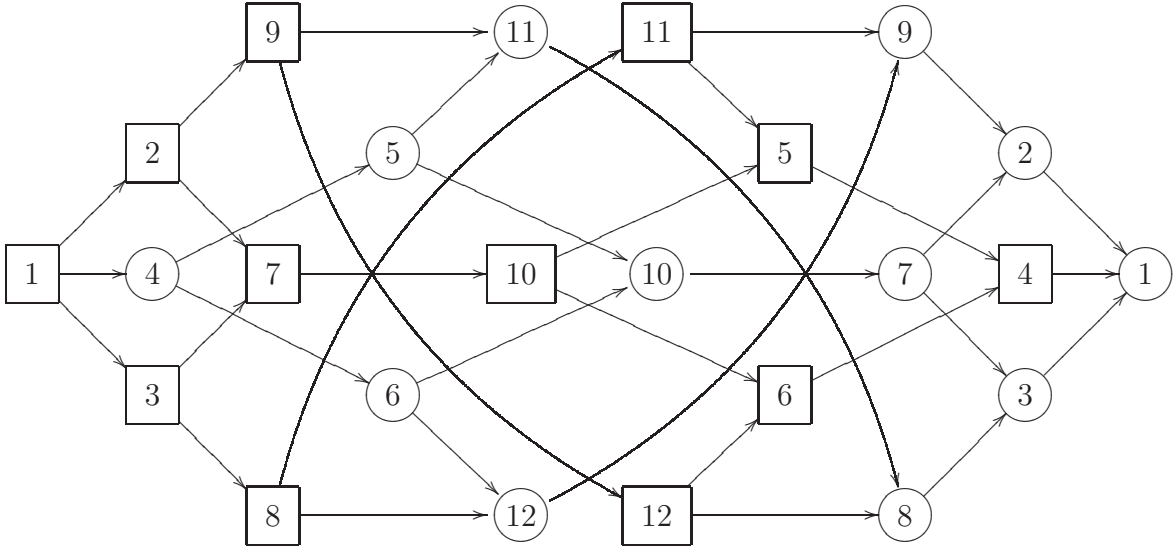
\[\xymatrix@C=0.9cm
{
&&*++[F]{9}\ar[rr]\ar@/_1cm/[ddddrrr]&&*++[o][F]{11}\ar@/^1cm/[ddddrrr]&*++[F]{11}\ar[rr]\ar[dr]&&*++[o][F]{9}\ar[dr]&&\\
&*++[F]{2}\ar[ur]\ar[dr]&&*++[o][F]{5}\ar[ur]\ar[drr]&&&*++[F]{5}\ar[drr]&&*++[o][F]{2}\ar[dr]&\\
*++[F]{1}\ar[r]\ar[ur]\ar[dr]&*++[o][F]{4}\ar[urr]\ar[drr]&*++[F]{7}\ar[rr]&&*++[F]{10}\ar[urr]\ar[drr]&*++[o][F]{10}\ar[rr]&&*++[o][F]{7}\ar[dr]\ar[ur]&*++[F]{4}\ar[r]&*++[o][F]{1}\\
&*++[F]{3}\ar[dr]\ar[ur]&&*++[o][F]{6}\ar[dr]\ar[urr]&&&*++[F]{6}\ar[urr]&&*++[o][F]{3}\ar[ur]&\\
&&*++[F]{8}\ar[rr]\ar@/^1cm/[uuuurrr]&&*++[o][F]{12}\ar@/_1cm/[uuuurrr]&*++[F]{12}\ar[ur]\ar[rr]&&*++[o][F]{8}\ar[ur]&&
}
\]

\begin{notation*}
Throughout this paper, let $\Lambda$ be a finite dimensional algebra over an algebraically closed field $K$,
and $\K_\Lambda:=\Kb(\proj\Lambda)$ denote the perfect derived category of $\Lambda$.
The $\Lambda$-dual is denoted by $(-)^*:=\Hom_?(-, \Lambda)$ for $?=\K_\Lambda$ or $\K_{\Lambda^\op}$.
\end{notation*}

\section{Results}

We say that an object $T$ of $\K_\Lambda$ is \emph{silting} (\emph{tilting}) if it satisfies $\Hom_{\K_\Lambda}(T, T[i])=0$ for every integer $i>0\ (i\neq0)$ and $\K_\Lambda=\thick T$.
Here, $\thick T$ stands for the smallest thick subcategory of $\K_\Lambda$ containing $T$.
We denote by $\silt\Lambda\ (\tilt\Lambda)$ the set of isomorphism classes of basic silting (tilting) objects in $\K_\Lambda$.

Let us recall silting mutation and a partial order on $\silt\Lambda$.

\begin{definitiontheorem}\cite[Theorem 2.11, 2.31, 2.35 and Definition 2.41]{AI}
\begin{enumerate}
\item Let $T$ be a silting object of $\K_\Lambda$ with decomposition $T=X\oplus M$.
Taking a minimal left $\add M$-approximation $f:X\to M'$ of $X$,
we construct a new object $\mu_X^-(T):=Y\oplus M$, where $Y$ is the mapping cone of $f$.
Then $\mu_X^-(T)$ is also silting, and we call it the \emph{left mutation} of $T$ with respect to $X$.
Dually, we define the \emph{right mutation} $\mu_X^+(T)$ of $T$ with respect to $X$.

\item For objects $T$ and $U$ of $\K_\Lambda$, we write $T\geq U$ if $\Hom_{\K_\Lambda}(T, U[i])=0$ for $i>0$.
Then $\geq$ gives a partial order on $\silt\Lambda$.

\item We construct the \emph{silting quiver} $\H$ of $\K_\Lambda$ as follows.
\begin{itemize}
\item The vertices of $\H$ are basic silting objects of $\K_\Lambda$;
\item We draw an arrow $T\to U$ if $U$ is a left mutation of $T$ with respect to an indecomposable direct summand.
\end{itemize}
Then $\H$ coincides with the Hasse quiver of the partially ordered set $\silt\Lambda$.
\end{enumerate}
\end{definitiontheorem}

We define a subset of $\silt\Lambda$ by
\[\twsilt\Lambda:=\{T\in\silt\Lambda\ |\ \Lambda\geq T\geq \Lambda[1] \}.\]
This bijectively corresponds to the poset of support $\tau$-tilting modules \cite[Theorem 3.2]{AIR}.

We say that $\Lambda$ \emph{admits an anti-automorphism} if there is a $K$-linear automorphism $\varsigma:\Lambda\to\Lambda$ satisfying $\varsigma(xy)=\varsigma(y)\varsigma(x)$, or equivalently if an algebra isomorphism $\sigma:\Lambda^\op\to\Lambda$ exists.
Here, $\Lambda^\op$ stands for the opposite algebra of $\Lambda$.
In this case, we obtain an equivalence $\K_{\Lambda^\op}\to\K_\Lambda$, also denoted by $\sigma$.

We now investigate that an anti-automorphism of $\Lambda$ induces a symmetry of $\silt\Lambda$/$\twsilt\Lambda$.

\begin{theorem}\label{symmetry}
Assume $\Lambda$ admits an anti-automorphism $\sigma$.
Then we have the following.
\begin{enumerate}
\item The functor $\S_\sigma:=\sigma\circ(-)^*$ induces an anti-automorphism of the poset $\silt\Lambda$.
\item Let $T$ be a silting object.
Then there is an algebra isomorphism $\End_{\K_\Lambda}(T)^\op\simeq\End_{\K_\Lambda}(\S_\sigma(T))$.
Moreover, if $\Gamma$ is derived equivalent to $\Lambda$, then so is $\Gamma^\op$; hence, $\Gamma$ and $\Gamma^\op$ are also derived equivalent.
\item The functor $S_\sigma:=[1]\circ\S_\sigma$ induces an anti-automorphism of the poset $\twsilt\Lambda$.
\end{enumerate}
\end{theorem}
\begin{proof}
(1)(3) It is evident that $(-)^*$ and $\sigma$ yield an anti-isomorphism $\silt\Lambda\to\silt\Lambda^\op$ and an isomorphism $\silt\Lambda^\op\to\silt\Lambda$, respectively.
Composing them makes an anti-automorphism of $\silt\Lambda$.
This immediately implies that $S_\sigma:=[1]\circ\S_\sigma$ is also an anti-automorphism of $\twsilt\Lambda$.

(2) Clearly, $\End_{\K_\Lambda}(\S_\sigma(T))\simeq \End_{\K_{\Lambda^\op}}(T^*)\simeq\End_{\K_\Lambda}(T)^\op$.
If $\Gamma$ is derived equivalent to $\Lambda$, then there is a tilting object $T$ of $\K_\Lambda$ with $\Gamma\simeq\End_{\K_\Lambda}(T)$.
Since $\S_\sigma(T)$ is also tilting, it is seen by (1) that $\Gamma^\op\simeq\End_{\K_\Lambda}(\S_\sigma(T))$ is derived equivalent to $\Lambda$.
\end{proof}

We discuss a benefit derived from the symmetry $S_\sigma$ of $\twsilt\Lambda$.

Let $P$ be an indecomposable projective $\Lambda$-module.
We define subsets of $\twsilt\Lambda$ by
\[\def\arraystretch{1.5}
\begin{array}{rl}
\T^-_P&:=\{T\in\twsilt\Lambda\ |\ \mu_P^-(\Lambda)\geq T\geq\Lambda[1] \}\ \mbox{and }  \\
\T^+_P &:= \{T\in\twsilt\Lambda\ |\ \Lambda\geq T\geq \mu_{P[1]}^+(\Lambda[1]) \}.
\end{array}\]
Denote by $X^i$ the $i$th term of a complex $X$. We make the following observation.

\begin{lemma}\label{dju0}
We have $\T^-_P=\{T\in\twsilt\Lambda\ |\ P\in \add T^{-1} \}$ and $\T^+_P=\{T\in\twsilt\Lambda\ |\ P\in \add T^{0} \}$.
In particular, $\T_P^-\sqcup\T_{P}^+=\twsilt\Lambda$.
\end{lemma}
\begin{proof}
Let $T\in\twsilt\Lambda$.
We know that $T$ is of the form $[T^{-1}\to T^0]$ with $T^{-1}, T^0\in\add \Lambda$.
By \cite[Lemma 2.25]{AI}, we have $\add T^{-1}\cap \add T^0=0$.
It is easily seen that $\add(T^{-1}\oplus T^0)=\add\Lambda$.
Now, we obtain from \cite[Theorem 2.35]{AI} (and its dual) that:
\begin{enumerate}[(i)]
\item $P\in \add T^{-1} \iff \mu_P^-(\Lambda)\geq T$;
\item $P\in \add T^{0} \iff T\geq \mu_{P[1]}^+(\Lambda[1])$.
\end{enumerate}
This completes the proof.
\end{proof}

The symmetry $S_\sigma$ is useful to analyze the cardinality of $\twsilt\Lambda$ as follows.

\begin{theorem}\label{double}
Let $e$ be a primitive idempotent of $\Lambda$ and put $P:=e\Lambda$.
Assume that $\Lambda$ admits an anti-automorphism $\sigma$.
If $\sigma(e)=e$,
then we have a bijection between $\T_P^-$ and $\T_P^+$, i.e., $\T_P^-\stackrel{{\rm anti}}{\simeq} \T_P^+$.
In particular, $|\twsilt\Lambda|=2\cdot|\T_P^-|=2\cdot|\T_P^+|$.
\end{theorem}
\begin{proof}
We see that $S_\sigma$ gives a one-to-one correspondence between $\T_P^-$ and $\T_{\S_\sigma(P)}^+$.
As $\S_\sigma(P)\simeq P$ by assumption,
the assertion follows from Lemma \ref{dju0}.
\end{proof}

Let $T:=[T_1\to T_0]$ be a 2-term silting object of $\K_\Lambda$; i.e., $T\in\twsilt\Lambda$, and $\E$ denote a complete list of pairwise orthogonal primitive idempotents of $\Lambda$.
Recall that the \emph{$g$-vector} $g_T$ of $T$ is the vector $(g_e)_{e\in\E}$ which is given by $g_e:=c_0^e-c_1^e$.
Here, $c_i^e$ stands for the multiplicity of $e\Lambda$ in $T_i$.

We immediately obtain the following corollary.

\begin{corollary}
Suppose that $\Lambda$ admits an anti-automorphism $\sigma$ satisfying $\sigma(e)=e$ for every primitive idempotent $e$ of $\Lambda$.
Then $S_\sigma$ reverses the directions of the $g$-vectors of all 2-term silting objects in $\K_\Lambda$.
\end{corollary}
\begin{proof}
Let $T:=[e_1\Lambda\to e_0\Lambda]$ be a 2-term silting object of $\K_\Lambda$, where $e_0$ and $e_1$ are idempotents of $\Lambda$.
Since any idempotent is fixed by $\sigma$, we observe that $S_\sigma$ sends $T$ to the 2-term silting object $[e_0\Lambda \to e_1\Lambda]$, which immediately tells us the fact that $g_{S_\sigma(T)}=-g_T$.
\end{proof}

\section{Applications and examples}

We explore when $\Lambda$ admits an anti-automorphism $\sigma$ with $\sigma(e)=e$ for some primitive idempotent $e$ of $\Lambda$, and give applications and examples of Theorem \ref{double}.

Let us start with enveloping algebras.

\begin{theorem}\label{envelopingalgebra}
The enveloping algebra $\Lambda^\op\otimes_K\Lambda$ has an anti-automorphism $(a\otimes b\mapsto b\otimes a)$ fixing the primitive idempotent $e\otimes e$ for a primitive idempotent $e$ of $\Lambda$.
In particular, there is a bijection between $\T_P^-$ and $\T_P^+$,
where $P:=(e\otimes e)\Lambda^\op\otimes_K\Lambda$.
\end{theorem}

Let $Q:=(Q_0, Q_1)$ be a (finite) quiver, where $Q_0$ and $Q_1$ are the sets of vertices and arrows, respectively.
For a vertex $v$ of $Q$, we denote by $e_v$ the primitive idempotent of $KQ$ corresponding to $v$.
The opposite quiver of $Q$ is denoted by $Q^\op$; that is, it consists of the same vertices as $Q$ and reversed arrows $a^*$ for arrows $a$ of $Q$, i.e., $a^*$ is obtained by swapping the source and target of $a$.
For an admissible ideal $I$ of $KQ$, reversing arrows makes the admissible ideal $I^\op$ of $KQ^\op$; for example, $ab\in I$ implies $b^*a^*\in I^\op$.

We consider the case that an isomorphism $\iota: Q^\op\to Q$ of quivers exists; $\iota$ gives rise to an algebra isomorphism $KQ^\op\to KQ$, which will be also written by $\iota$.

\begin{proposition}\label{qi}
Let $\Lambda$ be an algebra presented by a quiver $Q$ and an admissible ideal $I$ of $KQ$.
Suppose that there is an isomorphism $\iota: Q^\op\to Q$ of quivers satisfying $I^\op=\iota^{-1}(I)$ and fixing a vertex $v$; put $P:=e_v\Lambda$.
Then we have a bijection between $\T_P^-$ and $\T_P^+$.
In particular, $|\twsilt\Lambda|=2\cdot|\T_P^-|$.
\end{proposition}
\begin{proof}
As $I^\op=\iota^{-1}(I)$, we get isomorphisms
\[\Lambda^\op=(KQ/I)^\op\simeq KQ^\op/I^\op\stackrel{\iota}{\simeq}KQ/I=\Lambda;\]
write the composition by $\sigma:\Lambda^\op\to\Lambda$.
Since $\iota(v)=v$ by assumption, we have $\sigma(e_v)=e_v$.
Thus, the assertion follows from Theorem \ref{double}.
\end{proof}

\begin{example}
Let $\Lambda$ be the algebra given by the $A_n$-quiver $Q:1\xrightarrow{x} 2\xrightarrow{x}\cdots\xrightarrow{x} n$ and the admissible ideal $I=0$ or $I:=\langle x^r\rangle$ for some $r>0$.
We have an isomorphism $Q^\op\to Q$ of quivers which assigns $i\mapsto n-i+1\ (i\in Q_0)$ and $x^*\mapsto x\ (x\in Q_1)$.
The equalities $I^\op=\langle (x^*)^r\rangle=\iota^{-1}(I)$ imply that $\Lambda$ admits an anti-automorphism $\sigma$.
If $n$ is even, then we apply Theorem \ref{symmetry}.
If $n$ is odd, then the vertex $v:=\frac{n+1}{2}$ is fixed by $\sigma$, whence we can apply Proposition \ref{qi}; we get $\T_{e_v\Lambda}^-\stackrel{{\rm anti}}{\simeq} \T_{e_v\Lambda}^+$.
The following are the Hasse quivers of $\twsilt\Lambda$ for $n=2$ and $n=3$, in which $\xymatrix{*+[F]{-}}$ and $\xymatrix{*+[o][F]{-}}$ correspond and $\bullet$ is stable by $S_\sigma$.
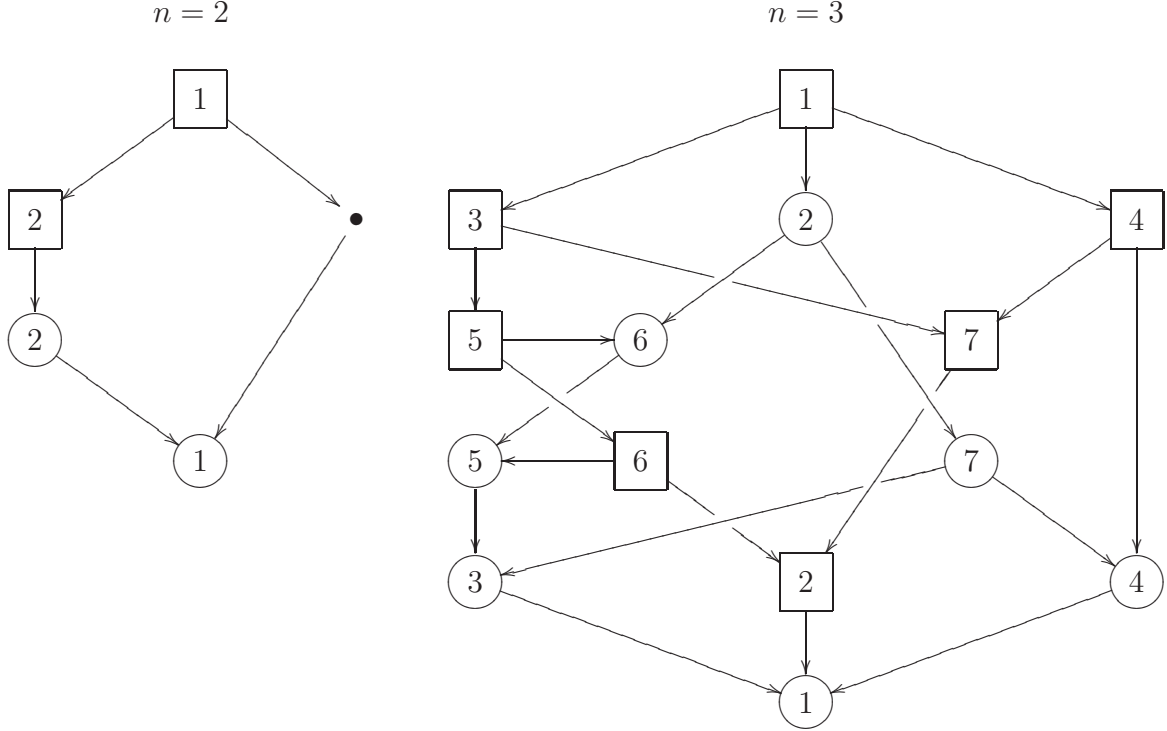
\[\begin{array}{c@{\hspace{1cm}}c}
n=2 & n=3 \\
\\
\xymatrix@C=1.5cm{
 & *++[F]{1} \ar[ld]\ar[rd] & \\
*++[F]{2} \ar[d] &           & \bullet \ar[ddl] \\
*++[o][F]{2} \ar[dr] &           & \\
          & *++[o][F]{1}           &
} &
\xymatrix@C=1.5cm{
&&*++[F]{1}\ar[dll]\ar[d]\ar[drr]&& \\
*++[F]{3}\ar[d]\ar[drrr]|\hole&&*++[o][F]{2}\ar[dl]\ar[ddr]|(0.4)\hole&&*++[F]{4}\ar[dl]\ar[ddd] \\
*++[F]{5}\ar[r]\ar[rd]&*++[o][F]{6}\ar[dl]|\hole&&*++[F]{7}\ar[ddl]|(0.25)\hole& \\
*++[o][F]{5}\ar[d]&*++[F]{6}\ar[dr]|\hole\ar[l]&&*++[o][F]{7}\ar[dr]\ar[dlll]|(0.2)\hole& \\
*++[o][F]{3}\ar[drr]&&*++[F]{2}\ar[d]&&*++[o][F]{4} \ar[dll] \\
&&*++[o][F]{1}&&
}
\end{array}\]
Here, in the RHS, $\xymatrix{*+[F]{-}}$ and $\xymatrix{*+[o][F]{-}}$ are the members of $\T_{P_2}^+$ and $\T_{P_2}^-$, respectively.
\end{example}

\subsection{Algebras presented by double quivers}\label{subsec:dq}

Recall that the \emph{double} quiver $\overline{Q}$ of $Q$ is the quiver constructed by $\overline{Q}_0:=Q_0$ and $\overline{Q}_1:=Q_1\sqcup\{a^*\ |\ a\in Q_1 \}$, where $a^*$ is obtained by swapping the source and target of $a$.
Clearly, the assignments $v\mapsto v\ (v\in Q_0)$, $a^*\mapsto a^*$ and $(a^*)^*\mapsto a\ (a\in Q_1)$ make an isomorphism $\iota: \overline{Q}^\op\to\overline{Q}$ of quivers; note that $\iota$ fixes all vertices.

Let us give examples of algebras presented by a double quiver.

\begin{example}
\begin{enumerate}
\item \cite{GP} The \emph{preprojective algebra} $\Pi_Q$ of a Dynkin quiver $Q$ is defined as the quotient $K\overline{Q}/\overline{I}$ of $K\overline{Q}$ by $\overline{I}:=\langle aa^*-a^*a\ |\ a\in Q_1 \rangle$.
Then, it is finite dimensional and selfinjective.
\item \cite[Example 1.6]{X}
Let $Q$ be a quiver and $I$ an admissible ideal of $KQ$.
For a path $p=a_1a_2\cdots a_\ell$ in $Q$, write $p^*:=a_\ell^*\cdots a_2^*a_1^*$; extending it linearly, we also use the terminology $p^*$ for a linear combination $p$ in $KQ$.
We define an ideal $\overline{I}$ of $K\overline{Q}$ which is generated by $p, p^*\ (p\in I)$ and $ab^*\ (a,b\in Q_1)$.
Then the algebra $\Lambda(Q, I):=K\overline{Q}/\overline{I}$ is finite dimensional.
If $Q$ contains no oriented cycle, then $\Lambda(Q,I)$ is a quasi-hereditary algebra with a duality.
\end{enumerate}
\end{example}

Now, an application of Proposition \ref{qi} is obtained.

\begin{theorem}\label{ppas}
Let $\Lambda=\Pi_Q$ for a Dynkin quiver $Q$ or $\Lambda(Q, I)$ for a quiver $Q$ and an admissible ideal $I$ of $KQ$.
Then we have a bijection between $\T_P^-$ and $\T_P^+$ for any indecomposable projective module $P$ of $\Lambda$.
In particular, $|\twsilt\Lambda|=2\cdot|\T_P^-|$.
\end{theorem}
\begin{proof}
We can easily check the equality $\overline{I}^\op=\iota^{-1}(\overline{I})$ holds, and apply Proposition \ref{qi}.
\end{proof}

\subsection{Cellular algebras}

Cellular algebras were introduced by Graham and Lehrer \cite{GL-cellular}.
An algebra $\Lambda$ is called \emph{cellular} if it admits a cellular basis; that is, a basis with certain nice multiplicative properties.
We refer to \cite{KX-cellular} for more details.
By the definition, each cellular basis of $\Lambda$ admits an involution $\sigma$; i.e., an anti-automorphism $\sigma$ of $\Lambda$ with $\sigma^2=1$.
It is shown in \cite[Proposition 5.1]{KX-cellular} that the involution $\sigma$ fixes all simples of a cellular algebra.
Hence, we have the following result.

\begin{theorem}\label{cellular-alg}
Let $\Lambda$ be a cellular algebra.
Then there exists a bijection between $\T_P^-$ and $\T_P^+$ for any indecomposable projective module $P$ of $\Lambda$.
In particular, $|\twsilt\Lambda|=2\cdot |\T_P^-|$.
\end{theorem}

Nowadays, a lot of interesting algebras have been found to be cellular,
for example, Ariki--Koike algebras, ($q$-)Schur algebras as well as various generalizations, block algebras of category $\mathcal{O}$, and various diagram algebras.
We hope that Theorem \ref{cellular-alg} will be useful to verify the finiteness of $|\twsilt\Lambda|$ for the aforementioned algebras,
especially, for Hecke algebras \cite{AS-Hecke}, Schur algebras \cite{W-schur}, etc.

\subsection{Symmetric algebras with radical cube zero}
We get the following result.

\begin{theorem}\label{rcz}
Let $\Lambda$ be a symmetric algebra with radical cube zero.
Then there exists a bijection between $\T_P^-$ and $\T_P^+$ for any indecomposable projective module $P$ of $\Lambda$.
In particular, $|\twsilt\Lambda|=2\cdot |\T_P^-|$.
\end{theorem}
\begin{proof}
By \cite[Proposition 3.3]{AA}, it turns out that the Gabriel quiver of $\Lambda$ is given by adding loops to the double quiver of a quiver $Q$; denote by $\widehat{Q}$ the quiver of $\Lambda$.
We also observe that $aa^*\neq0\neq a^*a$ for any arrow $a$ of $\widehat{Q}$ and $ab=0$ unless $b=a^*$ and $a=b^*$; if $a$ is an added loop, write $a^*=a$.
Thus, we get an isomorphism $\iota:\widehat{Q}^\op\to \widehat{Q}$ of quivers which fixes all vertices.

Let $i$ be a vertex of $\widehat{Q}$ and $a$ an arrow starting at $i$.
Since $\Lambda$ is symmetric, it is seen that $aa^*$ spans the socle of $P_i:=e_i\Lambda$ as a vector space.
Applying changes of basis, we have $aa^*=bb^*$ for every arrow $b$ of $\widehat{Q}$ starting from $i$.
Let $I$ denote the ideal of $K\widehat{Q}$ consisting of such relations; so $\Lambda\simeq K\widehat{Q}/I$.
Then, we obtain the equality $I^\op=\iota^{-1}(I)$, whence the assertion follows from Proposition \ref{qi}.
\end{proof}

\subsection{Selfinjective Nakayama algebras}

It is well-known that a selfinjective Nakayama algebra is presented by a cycle quiver $\xymatrix{\bullet \ar[r]_x & \bullet \ar[r]_x & \cdots \ar[r]_x & \bullet \ar@/_1pc/[lll]_x}$ with relations $x^r=0$ for some $r>0$.
Here is an easy application of Proposition \ref{qi}.

\begin{theorem}\label{Nakayama}
Let $\Lambda$ be a selfinjective Nakayama algebra and $P$ an indecomposable projective module of $\Lambda$.
Then we have a bijection between $\T_P^-$ and $\T_P^+$.
In particular, $|\twsilt\Lambda|=2\cdot|\T_P^-|$.
\end{theorem}

\begin{remark}
Let $\Lambda$ be a selfinjective Nakayama algebra given by a cycle quiver $Q$.
Whenever we choose a vertex $i$ of $Q$, one gets an isomorphism $Q^\op\to Q$ of quivers fixing $i$.
So, a bijection between $\T_{e_i\Lambda}^-$ and $\T_{e_i\Lambda}^+$ depends on the choice of vertices.
\end{remark}

\subsection{Group algebras}

Let $G$ be a finite group and $p$ the characteristic of $K$.
While the group algebra $KG$ is, in general, neither basic nor ring-indecomposable\footnote{It is well-known that if there is a normal $p$-subgroup of $G$ containing its centralizer, then $KG$ is ring-indecomposable; see \cite[Exercise V. 2. 10]{NT} for example.
}, it admits an anti-automorphism by $g\mapsto g^{-1}$; we can then apply Theorem \ref{symmetry} to $KG$.

The following situation enables us to apply Theorem \ref{double}.

\begin{theorem}\label{ga}
Let $G$ be a semidirect product $E\ltimes D$ of a $p'$-group $E$ (i.e., $p\nmid|E|$) on a $p$-group $D$.
Then there exists a primitive idempotent $e$ of $\Lambda:=KG$ such that $\T_{e\Lambda}^-$ bijectively corresponds to $\T_{e\Lambda}^+$.
In particular, $|\twsilt\Lambda|$ is double $|\T_{e\Lambda}^-|$.
\end{theorem}
\begin{proof}
As the argument above, we know that $\Lambda$ admits an anti-automorphism $\sigma\ (g\mapsto g^{-1})$.
Since $|E|$ is invertible in $K$, we put $e:=\frac{1}{|E|}\sum_{g\in E}g$; clearly, it is an idempotent fixed by $\sigma$.
It is seen that $e\Lambda=eKG=eKD\simeq KD$ (as $KD$-modules), which implies that $e$ is primitive.
Thus, we deduce the assertion from Theorem \ref{double}.
\end{proof}

We obtain an interesting observation.

\begin{corollary}\label{tSm}
Let $\Lambda$ be a $p$-block of $KG$ with a normal defect group $D$ and $E$ its inertial quotient. If $E$ has trivial Schur multiplier (i.e., $H^2(E, K^\times)=1$), then the number of 2-term silting objects is even if it is finite.
\end{corollary}
\begin{proof}
Thanks to K\"ulshammer's theorem \cite[Theorem A]{K}, we see that $\Lambda$ is Morita equivalent to the twisted group algebra $K^\alpha[E\ltimes D]$ for some 2-cocycle $\alpha$, which is just $K[E\ltimes D]$ by assumption.
Thus, we find out that $|\twsilt\Lambda|$ is even by Theorem \ref{ga}.
\end{proof}

It is known that groups of deficiency zero have the trivial Schur multiplier; see \cite{J}.
Here, the \emph{deficiency} of a group $G$ is defined to be the maximum of the integers $|X|-|R|$ for all presentations $G=\langle X\ |\ R \rangle$ of $G$, which is nonpositive if $G$ is a finite group.
Typical examples of deficiency-zero finite groups are cyclic groups $\langle g\ |\ g^n=1 \rangle$ and quaternion groups $\langle a,b\ |\ a^{2n}=1, a^n=b^2, ba=a^{-1}b \rangle=\langle a,b\ |\ bab=a^{n-1}, aba=b \rangle$.
Thus, the first example of Corollary \ref{tSm} should be the case that $D$ is cyclic; then, $E$ is automatically cyclic, $\Lambda$ is a symmetric Nakayama algebra \cite[Theorem 17.2]{Al}, and so $|\twsilt\Lambda|=\binom{2n}{n}$ (even), where $n:=|E|$ \cite[Corollary 2.29]{A}.
Moreover, the equality $|\twsilt\Lambda|=\binom{2n}{n}$ holds even if we drop the assumption of $D$ being normal in $G$;
then, $\Lambda$ is still a Brauer tree algebra \cite[Theorem 17.1]{Al}, whence the equality is obtained from \cite[Theorem 5.1]{AMN}.

\subsection{Trivial extension algebras}
The \emph{trivial extension} $T(\Lambda)$ of an algebra $\Lambda$ (by its minimal cogenerator $D\Lambda$) is defined to be $\Lambda\oplus D\Lambda$ as a $K$-vector space with multiplication given by $(a,f)\cdot (b,g):=(ab,ag+fb)$.
Here, $D$ denotes the $K$-dual.
We can easily verify that there is a one-to-one correspondence between simple modules of $\Lambda$ and $T(\Lambda)$;
so we use the same symbol $e$ as a primitive idempotent of $\Lambda$ and $T(\Lambda)$ (via the correspondence).

We state that a bisection of $\twsilt\Lambda$ can be extended to that of $\twsilt T(\Lambda)$.

\begin{theorem}\label{trivial-extension}
An anti-automorphism $\sigma$ of $\Lambda$ induces one on $T(\Lambda)$, say $\overline{\sigma}$.
If $\sigma$ fixes a primitive idempotent $e$ of $\Lambda$, then the corresponding idempotent $e$ of $T(\Lambda)$ is stable by $\overline{\sigma}$.
In the case, we have a bisection of $\twsilt T(\Lambda)$ with respect to $P:=eT(\Lambda)$.
\end{theorem}
\begin{proof}
Note that $T(\Lambda)^\op=T(\Lambda^\op)$. Since $\sigma^{-1}:\Lambda\to \Lambda^\op$ is an algebra isomorphism, we have a $K$-linear automorphism $t_\sigma:=\Hom_K(\sigma^{-1}, K):D(\Lambda^\op)\to D\Lambda$ of $D\Lambda$. For any $a, b\in \Lambda^\op$ and $f\in D(\Lambda^\op)$, we get equalities
\[\def\arraystretch{1.5}
\begin{array}{rl}
t_\sigma(a\bullet f\bullet b)(x) &= (a\bullet f\bullet b)(\sigma^{-1}(x))=f(b\bullet \sigma^{-1}(x)\bullet a)
 = f(\sigma^{-1}(\sigma(b) x\sigma(a))) \\
 &= t_\sigma(f)(\sigma(b)x\sigma(a))=(\sigma(a) t_\sigma(f)\sigma(b))(x).
\end{array}\]
Here, $\bullet$ stands for the multiplication or the action of $\Lambda^\op$. It turns out that
$$
t_\sigma(a\bullet f\bullet b)=\sigma(a)t_\sigma(f)\sigma(b).
$$

Now, we define a $K$-linear automorphism $\overline{\sigma}: T(\Lambda^\op)\to T(\Lambda)$ by $(a, f)\mapsto (\sigma(a), t_\sigma(f))$. Let us check that $\overline{\sigma}$ is an anti-automorphism of $T(\Lambda)$;
for any $a, b\in \Lambda^\op$ and $f, g\in D(\Lambda^\op)$,
\[
\def\arraystretch{1.5}
\begin{array}{rl}
\overline{\sigma}((a, f)\bullet (b,g))
 &= \overline{\sigma}(a\bullet b, a\bullet g+f\bullet b)   =(\sigma(a\bullet b), t_\sigma(a\bullet g+f\bullet b)) \\
 &= (\sigma(a)\sigma(b), \sigma(a)t_\sigma(g)+t_\sigma(f)\sigma(b)) \\
 &= (\sigma(a), t_\sigma(f))\cdot  (\sigma(b), t_\sigma(g))\\
 &= \overline{\sigma}(a, f)\cdot \overline{\sigma}(b, g).
\end{array}\]
Thus, the first assertion holds.
As the second assertion is clear, the last one immediately follows from Theorem \ref{double}.
\end{proof}

\begin{remark}
Theorem \ref{trivial-extension} does not imply that taking trivial extensions transmits the $\tau$-tilting finiteness.
In fact, the radical-square-zero selfinjective Nakayama algebra with 2 simple modules is $\tau$-tilting finite, but its trivial extension is not so.
\end{remark}

\subsection{Applying the main theorem twice}
In this subsection, we try applying Theorem \ref{double} twice in a row.
Let us show the following.

\begin{theorem}\label{twice}
Assume that $\Lambda$ is basic and admits an anti-automorphism $\sigma$ fixing a primitive idempotent $e$ of $\Lambda$; write $P:=e\Lambda$.
Let $P'$ be the mapping cone of a minimal left $\add(\Lambda/P)$-approximation of $P$;
that is, $\mu_P^-(\Lambda)=P'\oplus \Lambda/P$.
Putting $\Gamma:=\End_{\K_\Lambda}(\mu_P^-(\Lambda))$,
$e'$ denotes the idempotent of $\Gamma$ corresponding to $P'$.
Assume that the following hold:
\begin{enumerate}
\item $\mu_P^-(\Lambda)$ is tilting;
\item There is an anti-automorphism $\sigma'$ of $\Gamma$ satisfying $\sigma'(e')=e'$.
\end{enumerate}
Then, we have a poset isomorphism $\T_{P}^-\simeq\T_{e'\Gamma}^+$ and $|\twsilt\Lambda|=|\twsilt\Gamma|$.
\end{theorem}
\begin{proof}
As $\mu_P^-(\Lambda)$ is tilting, we identify $\twsilt\Gamma$ with $\{T\in\silt\Lambda\ |\ \mu_P^-(\Lambda)\geq T\geq \mu_P^-(\Lambda)[1] \}$.
By Lemma \ref{dju0}, we have an equality:
\[\def\arraystretch{1.3}
\begin{array}{l}
\{T\in\silt\Lambda\ |\ \mu_P^-(\Lambda)\geq T\geq \mu_P^-(\Lambda)[1] \} \\
=\{T\in\silt\Lambda\ |\ \mu_{P'}^-\mu_P^-(\Lambda)\geq T\geq \mu_P^-(\Lambda)[1] \}\sqcup\{T\in\silt\Lambda\ |\ \mu_P^-(\Lambda)\geq T\geq \Lambda[1]\},
\end{array}\]
in which the components of RHS have the same cardinality by Theorem \ref{double}.
Thus, the cardinality of LHS
in the equality is the double of that of $\T_P^-$,
which is equal to the cardinality of $\twsilt\Lambda$.
\end{proof}

We give two examples; one illustrates Theorem \ref{twice}, and the other explains that a derived equivalence does not necessarily preserve the cardinality of the poset $\twsilt(-)$ even if a given algebra is a symmetric algebra which admits an anti-automorphism fixing a primitive idempotent.

\begin{example}
Let $\Lambda$ be the algebra presented by the quiver with relations as follows:
\[\begin{array}{c@{\hspace{1cm}}c}
\vcenter{
\xymatrix{
 & 1 \ar[dr]^\beta &  \\
2 \ar[ru]^\alpha \ar@<2pt>[rr]^{\gamma^*} & & 3 \ar@<2pt>[ll]^\gamma
}
} &
\mbox{$
\begin{cases}
\ \beta\gamma\alpha=0=\gamma(\gamma^*\gamma)^3 \\
\ \alpha\beta=\gamma^*\gamma\gamma^*
\end{cases}
$}
\end{array}
\]
Note that $\Lambda$ is symmetric and admits an anti-automorphism which fixes the vertex 1 and switches the vertices 2 and 3.
Set $P_i:=e_i\Lambda$.
\begin{enumerate}
\item Let $T_1$ be the left mutation of $\Lambda$ with respect to $P_1$.
By hand, we can check that the endomorphism algebra $\Gamma_1$ of $T_1$ is given by the quiver with relations:
\[\begin{array}{c@{\hspace{1cm}}c}
\vcenter{
\xymatrix{
2 \ar@<2pt>[r]^\alpha & 1 \ar@<2pt>[r]^{\beta}\ar@<2pt>[l]^{\alpha^*} & 3 \ar@<2pt>[l]^{\beta^*}
}
} &
\mbox{$
\begin{cases}
\ \alpha\beta\beta^*\beta=\beta^*\beta\beta^*\alpha^*=\alpha\alpha^*=0 \\
\ \alpha^*\alpha=(\beta\beta^*)^2
\end{cases}
$}
\end{array}
\]
It is obtained that $\Gamma_1$ has an anti-automorphism fixing the vertex 1.
Thus, we derive from Theorem \ref{twice} that $\twsilt\Lambda$ and $\twsilt\Gamma_1$ has the same cardinality; it is illustrated by (anti-)isomorphisms $\T_{(P_1)_\Lambda}^+\stackrel{{\rm anti}}{\simeq}\T_{(P_1)_\Lambda}^-\simeq\T_{(P_1)_{\Gamma_1}}^+\stackrel{{\rm anti}}{\simeq}\T_{(P_1)_{\Gamma_1}}^-$.
Actually, $\twsilt\Lambda$ and $\twsilt\Gamma_1$ are finite sets and the numbers are 32 \cite[Theorem 2]{AHMW}.
\item Let $T_2$ be the left mutation of $\Lambda$ with respect to $P_2$.
We have the endomorphism algebra $\Gamma_2$ presented by the quiver with relations:
\[\begin{array}{c@{\hspace{1cm}}c}
\vcenter{
\xymatrix{
1 \ar@<2pt>[r]^\beta & 2 \ar@<2pt>[r]^{\gamma}\ar@(lu,ru)^\alpha\ar@<2pt>[l]^{\beta^*} & 3 \ar@<2pt>[l]^{\gamma^*}
}
} &
\mbox{$
\begin{cases}
\ \beta\gamma=\beta\beta^*=0=\gamma^*\beta^*=\gamma^*\alpha \\
\ \alpha\gamma=0 \\
\ \alpha^2=\beta^*\beta \\
\ \alpha^3=\gamma\gamma^*
\end{cases}
$}
\end{array}
\]
Unfortunately, the cardinality of $\twsilt\Gamma_2$ is 28 by \cite[Theorem 2]{AHMW}.
Since $\Gamma_2$ admits an anti-automorphism fixing the vertex 2, a similar argument as the proof of Theorem \ref{twice} explains that $\T_{(P_2)_\Lambda}^-\simeq\T_{(P_2)_{\Gamma_2}}^+\stackrel{{\rm anti}}{\simeq}\T_{(P_2)_{\Gamma_2}}^-$, and so we obtain $|\T_{(P_2)_\Lambda}^-|=14$ and $|\T_{(P_2)_\Lambda}^+|=18$.
(Note that $\T_{(P_2)_\Lambda}^-\stackrel{{\rm anti}}{\simeq}\T_{(P_3)_\Lambda}^+$; so, $|\T_{(P_3)_\Lambda}^-|=18$ and $|\T_{(P_3)_\Lambda}^+|=14$.)

When $U_3$ is the left mutation of $\Gamma_2$ with respect to $P_3$, the endomorphism algebra of $U_3$ is isomorphic to $\Lambda$.
This says that a derived equivalence does not necessarily preserve the number of $\twsilt(-)$, although $\Gamma_2$ is symmetric and admits an anti-automorphism fixing all vertices.
\end{enumerate}
\end{example}

There are some special derived equivalence classes of algebras for which the cardinalities of $2\silt(-)$ are constant, but the proofs are case by case for each algebra. Using Theorem 2.14, we may give an explicit example of such classes.

\begin{example}\label{A5A4}
Let $\Lambda$ be the multiplicity-free Brauer triangle algebra; that is, it is given by the quiver with relations as follows.
\[\begin{array}{c@{\hspace{1cm}}c}
\vcenter{
\xymatrix{
&1 \ar@<-2pt>[dl]_a\ar@<-2pt>[dr]|{a^*}& \\
2 \ar@<-2pt>[ur]|{a^*} \ar@<-2pt>[rr]_a &&3 \ar@<-2pt>[ul]_a \ar@<-2pt>[ll]_{a^*}
}
} &
\mbox{$
\begin{cases}
\ aa^*=a^*a \\
\ a^2=0=(a^*)^2
\end{cases}
$}
\end{array}
\]
We see that $\Lambda$ admits an anti-automorphism fixing every vertex; cf. Theorem \ref{rcz}.

Let $P:=e_1\Lambda$ and $\Gamma$ denote the endomorphism algebra of the left mutation $\mu_P^-(\Lambda)$; note that $\mu_P^-(\Lambda)$ is a tilting object in $\K_\Lambda$, and so $\Lambda$ and $\Gamma$ are derived equivalent.
By hand, we obtain that $\Gamma$ is presented by the quiver with relations:
\[\begin{array}{c@{\hspace{1cm}}c}
\vcenter{
\xymatrix{
2 \ar@<2pt>[r]^a & 1 \ar@<2pt>[r]^{b^*}\ar@<2pt>[l]^{a^*} & 3 \ar@<2pt>[l]^b
}
} &
\mbox{$
\begin{cases}
\ aa^*=0=bb^* \\
\ a^*ab^*b=b^*ba^*a
\end{cases}
$}
\end{array}
\]
Observe that $\Gamma$ admits an anti-automorphism fixing all vertices.

Thus, it turns out by Theorem \ref{twice} that $\twsilt\Lambda$ and $\twsilt\Gamma$ have the same cardinality;
actually, they are finite sets and the numbers are 32.
See $D(3K)$ and $D(3A)_1$ in Table 1 of \cite{EJR}.
Moreover, the class $\{\Lambda, \Gamma\}$ forms a derived equivalence class.
\end{example}


\section*{Acknowledgments}

The second author is grateful to Susumu Ariki for various conversations and lectures.
He also thanks Sota Asai and Kengo Miyamoto for useful discussions.

\section*{Conflict of Interest Statement}
Not applicable.



\begin{thebibliography}{AAAA}
\bibitem[Ad]{A}
{\sc T. Adachi},
The classification of $\tau$-tilting modules over Nakayama algebras.
{\it J. Algebra} {\bf 452} (2016), 227--262.

\bibitem[AA]{AA}
{\sc T. Adachi and T. Aoki},
The number of two-term tilting complexes over symmetric algebras with radical cube zero.
Preprint (2018), arXiv: 1805.08392.

\bibitem[AIR]{AIR}
{\sc T. Adachi, O. Iyama and I. Reiten},
$\tau$-tilting theory.
{\it Compos. Math.} {\bf 150}  (2014), no. 3, 415--452.

\bibitem[AHMW]{AHMW}
{\sc T. Aihara, T. Honma, K. Miyamoto and Q. Wang},
Report on the finiteness of silting objects.
{\it Proc. Edinb. Math. Soc. (2)} {\bf 64} (2021), no. 2, 217--233.

\bibitem[AI]{AI}
{\sc T. Aihara and O. Iyama},
Silting mutation in triangulated categories.
{\it J. Lond. Math. Soc. (2)} {\bf 85} (2012), no. 3, 633--668.

\bibitem[Al]{Al}
{\sc J. L. Alperin},
Local representation theory.
Cambridge Studies in Advanced Mathematics, {\bf 11}.
{\it Cambridge University Press, Cambridge}, 1986.

\bibitem[AMN]{AMN}
{\sc H. Asashiba, Y. Mizuno and K. Nakashima},
Simplicial complexes and tilting theory for Brauer tree algebras.
{\it J. Algebra} {\bf 551} (2020), 119--153.

\bibitem[AS]{AS-Hecke}
{\sc S. Ariki and L. Speyer},
Schurian-finiteness of blocks of type A Hecke algebras
Preprint (2022), arXiv: 2112.11148.

\bibitem[EJR]{EJR}
{\sc F. Eisele, G. Janssens and T. Raedschelders},
A reduction theorem for $\tau$-rigid modules.
{\it Math. Z.} {\bf 290} (2018), no. 3--4, 1377--1413.

\bibitem[GL]{GL-cellular}
{\sc J. J. Graham and G. I. Lehrer},
Cellular algebras.
{\it Invent. Math.} {\bf 123} (1996), no. 1, 1--34.

\bibitem[GP]{GP}
{\sc I. M. Gelfand and V. A. Ponomarev},
Model algebras and representations of graphs.
{\it Funktsional. Anal. i Prilozhen.} {\bf 13} (1979), no. 3, 1--12.

\bibitem[J]{J}
{\sc D. L. Johnson},
Presentations of groups. London Mathematical Society Student Texts, {\bf 15}.
{\it Cambridge University Press, Cambridge}, 1976.

\bibitem[K]{K}
{\sc B. Kulshammer},
Crossed products and blocks with normal defect groups.
{\it Comm. Algebra} {\bf 13} (1985), no. 1, 147--168.

\bibitem[KX]{KX-cellular}
{\sc S. Konig and C. Xi},
On the structure of cellular algebras.
{\it Algebra and modules, II}, 365--386,
CMS Conf. Proc., {\bf 24}, {\it Amer. Math. Soc., Providence, RI}, 1998.

\bibitem[M]{Mizuno-preprojective-alg}
{Y. Mizuno},
Classifying $\tau$-tilting modules over preprojective algebras of Dynkin type.
{\it Math. Z.}, {\bf 277} (2014), no. 3-4, 665--690.

\bibitem[NT]{NT}
{\sc H. Nagao and Y. Tsushima},
Representations of finite groups.
{\it Academic Press, Inc., Boston, MA}, 1989.

\bibitem[W]{W-schur}
{\sc Q. Wang},
On $\tau$-tilting finiteness of the Schur algebra.
{\it J. Pure Appl. Algebra}, {\bf 226 (1)} (2022), 106818.

\bibitem[X]{X}
{\sc C. Xi},
Quasi-hereditary algebras with a duality.
{\it J. reine angew. Math.} {\bf 449} (1994), 201--215.

\end{thebibliography}
\end{document}